\documentclass[11pt]{article}
\usepackage{amsmath,amsfonts,amssymb,amsthm}
\usepackage{geometry}
\usepackage{enumitem}
\usepackage{fancyhdr}
\usepackage{hyperref}
\hypersetup{colorlinks=false}
\usepackage[
backend=biber,
style=numeric,
sorting=nyt
]{biblatex}

\allowdisplaybreaks
 
\newtheorem{theorem}{Theorem}[section]
\newtheorem{corollary}[theorem]{Corollary}
\newtheorem{lemma}[theorem]{Lemma}
\newtheorem{proposition}[theorem]{Proposition}

\theoremstyle{definition}
\newtheorem{definition}[theorem]{Definition}
\newtheorem{example}[theorem]{Example}
\newtheorem{remark}[theorem]{Remark}

\title{The integral shuffle algebra and the $K$-theory of the Hilbert scheme of points in $\mathbb{A}^2$}
\author{Frank Wang}
\date{\vspace{-5ex}}

\begin{document}

\maketitle
\begin{abstract}

    We examine the shuffle algebra defined over the ring $\mathbf{R} = \mathbb{C}[q_1^{\pm 1}, q_2^{\pm 1}]$, also called the integral shuffle algebra, which was found by Schiffmann and Vasserot to act on the equivariant $K$-theory of the Hilbert scheme of points in the plane. We find that the modules of 2 and 3 variable elements of the integral shuffle algebra are finitely generated and prove a necessary condition for an element to be in the integral shuffle algebra for arbitrarily many variables.
\end{abstract}
\section{Introduction}
\subsection{Motivation}
The shuffle presentation of the quantum toroidal algebra is studied by Feigin and Odesskii in \cite{MR1873567}, where they considered a ``shuffle" product, denoted $*$, of two symmetric rational functions $P(z_1,\dots,z_k)$ and $Q(z_1, \dots,z_l)$ that outputs a rational function in $k+l$ variables in the following form: $$P(z_1,\dots,z_k) * Q(z_1, \dots,z_l) = \frac{1}{k!l!}\text{Sym}\bigg[P(z_1,\dots,z_k)Q(z_{k+1},\dots,z_{k+l})\prod_{1 \leq i \leq k < j \leq k + l}\omega(z_1,z_j)\bigg].$$ In this paper, we examine the integral shuffle algebra $A^\mathbf{R}$ defined over the ring $\mathbf{R} = \mathbb{C}[q_1^{\pm 1},q_2^{\pm 1}]$.

A closely related shuffle algebra to the integral shuffle algebra is the fractional shuffle algebra, defined over the field $\mathbb{C}(q_1,q_2)$. The structure of this algebra was studied by Negut \cite{negut2014shuffle}, in which it was found that the wheel conditions formed necessary and sufficient conditions for a rational function to be in the algebra.

\subsection{Relation to other work}
The Hilbert scheme of $n$ points in the plane Hilb$_n$ is defined as the set of ideals $I$ of $\mathbb{C}[x,y]$ such that $\mathbb{C}[x,y]/I$ has dimension $n$ as a vector space over $\mathbb{C}$. Schiffmann and Vasserot \cite{schiffmann2013elliptic} showed that the integral shuffle algebra acts on the equivariant $K$-theory of the Hilbert scheme of points in the plane. The fractional shuffle algebra was also shown to act on a localization of the equivariant $K$-theory of the Hilbert scheme of points in the plane.

In \cite{gorsky2016flag} Gorsky-Negut-Rasmussen proposed a monoidal functor from the monoidal category of coherent sheaves on the flag Hilbert scheme to the (non-symmetric) monoidal category of Soergel bimodules. Moreover, Gorsky-Negut \cite{gorsky2015refined} and Oblomkov-Rozansky \cite{oblomkov2018knot} related the Hilbert scheme of points in
$\mathbb{A}^2$ to knot invariants. 

Knot theory examines the way in which curves and surfaces can be tied in knots. This kind of knotting is not only relevant to understanding topology, but has recently become significant in the study of DNA. Confined to a small space, long strands of DNA naturally become knotted, and certain processes depend upon an understanding of the complexity of these knots. Applications of this project include effective ways of measuring different types of complexity of knots.
\subsection{Description of the results}
In this paper, we study the structure of the integral shuffle algebra $A^\mathbf{R}$. We prove the following fundamental properties of the subsets $A^\mathbf{R}_k, k\in\mathbb{Z}^+$ of the algebra:

\begin{theorem} [Theorem \ref{thm3:3}] \label{thm1:1}
    Let $A^\mathbf{R}_k$ be the subset of the integral shuffle algebra consisting of functions in $k$ variables. Then the following hold:
    \begin{enumerate}[label = (\alph*)]
        \item $A^\mathbf{R}_k$ is a module over $\mathbf{R}[z_1^{\pm 1},\dots,z_k^{\pm 1}]$.
        \item $A^\mathbf{R}_2$ is generated by $z_1*z_1^0$ and $z_1^0*z_1^0$ as a module over $\mathbf{R}[z_1^{\pm 1},z_2^{\pm 1}]$.
        \item $A^\mathbf{R}_3$ is generated by the elements $z_1^{d_1}*z_1^{d_2}*z_1^0$ for $0\leq d_1\leq 2$, $0\leq d_2\leq 1$ as a module over $\mathbf{R}[z_1^{\pm 1},z_2^{\pm 1},z_3^{\pm 1}]$.
    \end{enumerate}
\end{theorem}

We also look at the general structure of the integral shuffle algebra and prove the following necessary condition for an element to be in the integral shuffle algebra:

\begin{theorem} [Theorem \ref{thm4:2}] \label{thm1:2}
    $A^{\mathbf{R}}_k$ is contained in the ideal $$(2qz_1^2 - (1 + q_1 + q_2 - 2q + q_1q + q_2q + q^2)z_1z_2 + 2qz_2^2, (1 - q_1)(1 - q_2)(1 - q)(z_1 + z_2))$$of $\mathbf{R}[z_1^{\pm 1}, \dots, z_k^{\pm 1}]$ for $k \geq 2$.
\end{theorem}

These results will help us better understand the structure of the Hilbert scheme of points in the plane and may also provide insight into the structure of the Hilbert scheme of points in an arbitrary surface.
\subsection{Structure of the paper}

In Section \ref{sec2}, we provide basic definitions and examples for the integral shuffle algebra and note a proposition due to Negut. In Section \ref{sec3}, we describe the generators of $A^\mathbf{R}_k$ and use this approach to prove Theorem \ref{thm1:1}. We also show the limitations of this approach. In Section \ref{sec4}, we present some necessary conditions for an element to be in the integral shuffle algebra in the form of membership of an ideal, prove Theorem \ref{thm1:2}, and propose two open problems for further investigation of the integral shuffle algebra.

\section{The integral shuffle algebra}\label{sec2}
Throughout this paper, we will work over the ring $\mathbf{R} = \mathbb{C}[q_1^{\pm 1}, q_2^{\pm 1}]$ and we will denote $q=q_1q_2$. Also, let us define $$V_k = \text{Sym}_\mathbf{R}[z_1^{\pm 1}, \dots, z_k^{\pm 1}],$$$$V = \bigoplus_{k \geq 0} V_k,$$where $\text{Sym}_\mathbf{R}[z_1^{\pm 1}, \dots, z_k^{\pm 1}]$ is the set of symmetric functions in $z_1^{\pm 1}, \dots, z_k^{\pm 1}$ over $\mathbf{R}$. For the sake of clarity, let us also denote $1_d$ as the element 1 in $V_d$. For example, $1_1$ can be interpreted as $z_1^0$ and $1_2$ can be interpreted as $z_1^0z_2^0$. Next, we introduce a few definitions:

\begin{definition}
    We define the \textbf{shuffle product} $*:V_k\times V_l\rightarrow V_{k+l}$ as the product that takes $P\in V_k,Q\in V_l$ to $$(P*Q)(z_1, \dots, z_{k+l}) = \frac{1}{k!l!}\text{Sym} \bigg[P(z_1, \dots, z_k)Q(z_{k + 1}, \dots, z_{k + l})\prod_{1 \leq i \leq k < j \leq k + l}\omega(z_i,z_j)\bigg],$$where Sym denotes the symmetric sum, i.e. $$\text{Sym}(P(z_1,\dots,z_k)) = \sum_{\sigma\in S_k}P(z_{\sigma(1)},\dots,z_{\sigma(k)}),$$and$$\omega(z_i,z_j)=\frac{(z_i - qz_j)(z_j - q_1z_i)(z_j - q_2z_i)}{z_i - z_j}.$$
\end{definition}

The shuffle product is associative, as noted in \cite{MR1873567}. Throughout this paper, we will adopt the notation used in \cite{negut2014shuffle} and use an asterisk $*$ to denote the shuffle product and parentheses to denote standard multiplication.

\begin{example}
    Let us compute $1_1*1_1$. We have 
    \begin{align*}
        1_1*1_1 & =\text{Sym} \bigg[z_1^0z_2^0\prod_{1 \leq i \leq 1 < j \leq 2}\frac{(z_i - qz_j)(z_j - q_1z_i)(z_j - q_2z_i)}{z_i - z_j}\bigg]\\
        & =\text{Sym} \bigg[\frac{(z_1 - qz_2)(z_2 - q_1z_1)(z_2 - q_2z_1)}{z_1 - z_2}\bigg].
    \end{align*}
    We can now expand out the Sym:
    \begin{align*}
        1_1*1_1 &= \frac{(z_1 - qz_2)(z_2 - q_1z_1)(z_2 - q_2z_1)}{z_1 - z_2} + \frac{(z_2 - qz_1)(z_1 - q_1z_2)(z_1 - q_2z_2)}{z_2 - z_1} \\
        & = \frac{(z_1 - qz_2)(z_2 - q_1z_1)(z_2 - q_2z_1) - (z_2 - qz_1)(z_1 - q_1z_2)(z_1 - q_2z_2)}{z_1 - z_2}.
    \end{align*}
    The rest of the derivation consists only of basic computation, so we skip to the final form:
    $$1_1*1_1=2qz_1^2 - (1 + q_1 + q_2 - 2q + q_1q + q_2q + q^2)z_1z_2 + 2qz_2^2.$$
\end{example}
\begin{definition}
    The \textbf{integral shuffle algebra} $A^{\mathbf{R}} \subset V$ is the algebra over $\mathbf{R}$ generated by elements of the form $z_1^d$, $d\in \mathbb{Z}$, with the product of the algebra being the shuffle product. A \textbf{shuffle element} is an element of the integral shuffle algebra. We will denote $A^{\mathbf{R}}_k = A^{\mathbf{R}} \cap V_k$ and say that a shuffle element in $A^{\mathbf{R}}_k$ has \textbf{degree} $k$.
    
\end{definition}
\begin{example}
    $A^{\mathbf{R}}_1$ is the set of Laurent polynomials in $z_1$ with coefficients in $\mathbf{R}$.
\end{example}

The next proposition was proved by Negut in \cite{negut2014shuffle} for the shuffle algebra over $\mathbb{C}(q_1,q_2)$, and the proof also holds for the integral shuffle algebra.

\begin{proposition}[\cite{negut2014shuffle}]\label{prop2:5}
    All shuffle elements are symmetric Laurent polynomials that satisfy the \textbf{wheel conditions}:$$p(z_1, z_2, z_3, \dots) = 0\text{ whenever }\bigg\{\frac{z_1}{z_2}, \frac{z_2}{z_3}, \frac{z_3}{z_1}\bigg\} = \bigg\{q_1, q_2, \frac{1}{q}\bigg\}.$$
\end{proposition}

\section{Generators for the integral shuffle algebra}\label{sec3}
In describing the elements of $A^{\mathbf{R}}$, it is easy to see that since the shuffle product of two shuffle elements $p(z_1, \dots, z_k)$ and $p'(z_1, \dots, z_{k'})$ has $k + k'$ variables, any $A^\mathbf{R}_k$ is completely determined by the shuffle products of elements in $A_l$ with $l < k$. We can use this to describe the generators of any $A^{\mathbf{R}}_k$, which is useful in describing the structure of the integral shuffle algebra as a whole.

\begin{proposition}\label{prop3:1}
    For any $k\in \mathbb{N}$, $A^{\mathbf{R}}_k$ is generated by elements of the form $z_1^{d_1} * z_1^{d_2} * \cdots * z_1^{d_k}$ as a module over $\mathbf{R}$, where $d_i\in\mathbb{Z}$.
\end{proposition}
\begin{proof}
    We will prove this using induction. For our base case, we have that $A^{\mathbf{R}}_1$ is generated by the elements $z_1^d$ by the definition of the integral shuffle algebra. Now, assume that the hypothesis holds for all $l < k$. Since any shuffle element in $A^{\mathbf{R}}_k$ is a linear combination of shuffle products of elements with fewer variables, we can simply consider those shuffle products. Consider the shuffle product $$P(z_1, \dots, z_l) * Q(z_1, \dots, z_{k - l})$$$$=\frac{1}{k!l!}\text{Sym} \bigg[P(z_1, \dots, z_l)Q(z_{l + 1}, \dots, z_k)\prod_{1 \leq i \leq l < j \leq k}\omega(z_i,z_j)\bigg].$$
    By the inductive hypothesis, both $P$ and $Q$ can be written as a linear combination of the form 
    \begin{align*}
        P(z_1, \dots, z_l) & = \sum_{i = 1}^m p_i(z_1^{d_{i, 1}} * z_1^{d_{i, 2}} *\cdots *z_1^{d_{i, l}}),\\
        Q(z_{l + 1}, \dots, z_k) & = \sum_{j = 1}^n q_j(z_1^{d_{j, l + 1}} * z_1^{d_{j, l + 2}} *\cdots *z_1^{d_{j, k}}).
    \end{align*}
    Substituting this into the above expression, we get 
    $$P(z_1, \dots, z_l) * Q(z_1, \dots, z_{k - l})$$
    \begin{align*}
        & = \frac{1}{k!l!}\text{Sym} \Bigg[\sum_{i = 1}^m p_i(z_1^{d_{i, 1}} * z_1^{d_{i, 2}} *\cdots *z_1^{d_{i, l}})\sum_{j = 1}^n q_j(z_1^{d_{j, l + 1}} * z_1^{d_{j, l + 2}} *\cdots *z_1^{d_{j, k}})\prod_{1 \leq i \leq l < j \leq k}\omega(z_i,z_j)\Bigg]\\
        & = \frac{1}{k!l!}\sum_{i = 1}^m \sum_{j = 1}^n p_iq_j\text{Sym} \Bigg[(z_1^{d_{i, 1}} * z_1^{d_{i, 2}} *\cdots *z_1^{d_{i, l}})(z_1^{d_{j, l + 1}} * z_1^{d_{j, l + 2}} *\cdots *z_1^{d_{j, k}})\prod_{1 \leq i \leq l < j \leq k}\omega(z_i,z_j)\Bigg]\\
        & = \frac{1}{k!l!}\sum_{i = 1}^m \sum_{j = 1}^n p_iq_j(z_1^{d_{i, 1}} * z_1^{d_{i, 2}} *\cdots *z_1^{d_{i, l}}) * (z_1^{d_{j, l + 1}} * z_1^{d_{j, l + 2}} *\cdots *z_1^{d_{j, k}})\\
        & = \frac{1}{k!l!}\sum_{i = 1}^m \sum_{j = 1}^n p_iq_j(z_1^{d_{i, 1}} * z_1^{d_{i, 2}} *\cdots *z_1^{d_{i, l}} * z_1^{d_{j, l + 1}} * z_1^{d_{j, l + 2}} *\cdots *z_1^{d_{j, k}}),
    \end{align*}
    so the product can be written as a linear combination of the generators in question. Conversely, all of the generators are contained in the integral shuffle algebra by definition.
\end{proof}

The above proposition gives us a complete picture of $A^{\mathbf{R}}_k$ as a module over $\mathbf{R}$. However, the more interesting question is whether $A^{\mathbf{R}}_k$ forms a module over $V_k$, and whether we can say anything about this module. A description of $A^{\mathbf{R}}_k$ over $V_k$ would give us a clearer and more useful description of the integral shuffle algebra as a whole. In this section, we will explore this by examining the integral shuffle algebra via its generators. The next lemma will be the main tool which we use.

\begin{lemma}\label{lemma3:2}
    The following relations hold:
    \begin{enumerate}[label = (\alph*)]
        \item $(z_1z_2\dots z_k)^n(z_1^{d_1} * z_1^{d_2} * \cdots * z_1^{d_k}) = z_1^{d_1 + n} * z_1^{d_2 + n} * \cdots * z_1^{d_k + n}$\label{lemma3:2:a}
        \item $(z_1^n + z_2^n + \dots + z_k^n)(z_1^{d_1} * z_1^{d_2} * \cdots * z_1^{d_k}) = \sum_{i = 1}^k (z_1^{d_1} * \dots * z_1^{d_{i - 1}} * z_1^{d_i + n} * z_1^{d_{i + 1}} * \dots * z_1^{d_k})$\label{lemma3:2:b}
    \end{enumerate}
\end{lemma}
\begin{proof}
    \ref{lemma3:2:a} We will prove this using induction. For $k = 1$, we have $z_1^n(z_1^{d_1}) = z_1^{d_1 + n}$ by the multiplication of polynomials. For general $k$, we have $$(z_1z_2\dots z_k)^n(z_1^{d_1} * z_1^{d_2} * \cdots * z_1^{d_k})$$$$= (z_1z_2\dots z_k)^n\frac{1}{(k-1)!}\text{Sym} \bigg[(z_1^{d_1} * z_1^{d_2} * \cdots * z_1^{d_{k - 1}})z_k^{d_k}\prod_{1 \leq i \leq k - 1}\omega(z_i,z_k)\bigg].$$Since $(z_1z_2\dots z_k)^n$ is symmetric in $\{z_1, \dots, z_k\}$, we can distribute it into the symmetric operator to obtain the following:
    $$(z_1z_2\dots z_k)^n(z_1^{d_1} * z_1^{d_2} * \cdots * z_1^{d_k})$$
    \begin{align*}
        & = \frac{1}{(k-1)!}\text{Sym} \bigg[(z_1z_2\dots z_k)^n(z_1^{d_1} * z_1^{d_2} * \cdots * z_1^{d_{k - 1}})z_k^{d_k}\prod_{1 \leq i \leq k - 1}\omega(z_i,z_k)\bigg]\\
        & = \frac{1}{(k-1)!}\text{Sym} \bigg[(z_1^{d_1 + n} * z_1^{d_2 + n} * \cdots * z_1^{d_{k - 1} + n})z_k^{d_k + n}\prod_{1 \leq i \leq k - 1}\omega(z_i,z_k)\bigg]\\
        & = z_1^{d_1 + n} * z_1^{d_2 + n} * \cdots * z_1^{d_k + n},
    \end{align*}
    where $(z_1z_2\dots z_{k - 1})^n(z_1^{d_1} * z_1^{d_2} * \cdots * z_1^{d_{k - 1}}) = z_1^{d_1 + n} * z_1^{d_2 + n} * \cdots * z_1^{d_{k - 1} + n}$ by the inductive hypothesis.
    
    \ref{lemma3:2:b} We will also prove this using induction. For $k = 1$, we have $z_1^n(z_1^{d_1}) = z_1^{d_1 + n}$. For general $k$, we have $$(z_1^n + z_2^n + \dots + z_k^n)(z_1^{d_1} * z_1^{d_2} * \cdots * z_1^{d_k})$$
    \begin{align*}
        & = (z_1^n + z_2^n + \dots + z_k^n)\frac{1}{(k-1)!}\text{Sym} \bigg[(z_1^{d_1} * z_1^{d_2} * \cdots * z_1^{d_{k - 1}})z_k^{d_k}\prod_{1 \leq i \leq k - 1}\omega(z_i,z_k)\bigg]\\
        & = \frac{1}{(k-1)!}\text{Sym} \bigg[(z_1^n + z_2^n + \dots + z_k^n)(z_1^{d_1} * z_1^{d_2} * \cdots * z_1^{d_{k - 1}})z_k^{d_k}\prod_{1 \leq i \leq k - 1}\omega(z_i,z_k)\bigg]\\
        & = \frac{1}{(k-1)!}\text{Sym} \bigg[((\sum_{i = 1}^{k - 1} z_1^{d_1} * \dots * z_1^{d_i + n} * \dots * z_1^{d_{k - 1}})z_k^{d_k} \\
        &\qquad + (z_1^{d_1} * z_1^{d_2} * \cdots * z_1^{d_{k - 1}})z_k^{d_k + n})\prod_{1 \leq i \leq k - 1}\omega(z_i,z_k)\bigg]\\
        & = \frac{1}{(k-1)!}\text{Sym} \bigg[((\sum_{i = 1}^{k - 1} z_1^{d_1} * \dots * z_1^{d_i + n} * \dots * z_1^{d_{k - 1}})z_k^{d_k})\prod_{1 \leq i \leq k - 1}\omega(z_i,z_k)\bigg]\\
        &\qquad + \frac{1}{(k-1)!}\text{Sym} \bigg[((z_1^{d_1} * z_1^{d_2} * \cdots * z_1^{d_{k - 1}})z_k^{d_k + n})\prod_{1 \leq i \leq k - 1}\omega(z_i,z_k)\bigg]\\
        & = \sum_{i = 1}^{k - 1} \frac{1}{(k-1)!}\text{Sym} \bigg[((z_1^{d_1} * \dots * z_1^{d_i + n} * \dots * z_1^{d_{k - 1}})z_k^{d_k})\prod_{1 \leq i \leq k - 1}\omega(z_i,z_k)\bigg]\\
        &\qquad + (z_1^{d_1} * z_1^{d_2} * \cdots * z_1^{d_{k - 1}} * z_1^{d_k + n})\\
        & = (\sum_{i = 1}^{k - 1} z_1^{d_1} * \dots * z_1^{d_i + n} * \dots * z_1^{d_{k - 1}} * z_1^{d_k}) + (z_1^{d_1} * z_1^{d_2} * \cdots * z_1^{d_{k - 1}} * z_1^{d_k + n})\\
        & = \sum_{i = 1}^k (z_1^{d_1} * \dots * z_1^{d_{i - 1}} * z_1^{d_i + n} * z_1^{d_{i + 1}} * \dots * z_1^{d_k}),
    \end{align*}
    where $(z_1^n + z_2^n + \dots + z_{k - 1}^n)(z_1^{d_1} * z_1^{d_2} * \cdots * z_1^{d_{k - 1}}) = \sum_{i = 1}^{k - 1} (z_1^{d_1} * \dots * z_1^{d_{i - 1}} * z_1^{d_i + n} * z_1^{d_{i + 1}} * \dots * 
    z_1^{d_{k - 1}})$ by the inductive hypothesis.
\end{proof}

With this lemma we can prove our first main theorem:

\begin{theorem}\label{thm3:3}
    The following statements hold:
    \begin{enumerate}[label = (\alph*)]
        \item $A^\mathbf{R}_k$ is a module over $V_k$ for all $k$.\label{thm3:3:a}
        \item As a $V_2$-module, $A^\mathbf{R}_2$ is generated by $1_1 * 1_1$ and $z_1 * 1_1$.\label{thm3:3:b}
        \item As a $V_3$-module, $A^{\mathbf{R}}_3$ is generated by $z_1^{d_1} * z_1^{d_2} * 1_1$ for $0\leq d_1\leq 2, 0\leq d_2\leq 1$.\label{thm3:3:c}
    \end{enumerate}
\end{theorem}
\begin{proof}
    \ref{thm3:3:a} Note that every element of $V_k$ can be written in the form $$\frac{p(z_1, \dots, z_k)}{(z_1\dots z_k)^n}$$where $p(z_1, \dots,z_k)$ is a symmetric polynomial. Now, by the theory of symmetric polynomials, $p(z_1, \dots,z_k)$ can be written as a linear combination of products of polynomials of the form $$z_1^d + \dots + z_k^d,\quad (z_1\dots z_k)^d.$$
    Therefore, by Lemma \ref{lemma3:2} the product of any element of $V_k$ and any element of $A_k^\mathbf{R}$ is in $A^{\mathbf{R}}_k$, so $A^{\mathbf{R}}_k$ is a module over $V_k$.
    
    \ref{thm3:3:b} Observe that the conditions for generating an element $z_1^{d_2}*z_1^{d_3}$ in $A^\mathbf{R}_2$ using Lemma \ref{lemma3:2} are necessarily more lenient than the conditions for generating the element $z_1^{d_1}*z_2^{d_2}*z_3^{d_3}$ in $A^\mathbf{R}_3$ using Lemma \ref{lemma3:2}. Therefore, part \ref{thm3:3:c} of this theorem implies that all elements of the form $z_1^{d_2}*z_1^{d_3}$ are generated by $z_1*1_1$ and $1_1*1_1$ by simply observing the last 2 factors of the elements $z_1^{d_1}*z_2^{d_2}*z_3^{d_3}$. By Proposition \ref{prop3:1}, this implies that all of $A^\mathbf{R}_2$ is generated by $z_1*1_1$ and $1_1*1_1$.
    
    \ref{thm3:3:c} This proof will be split into two parts: The first part will be a list of computations to show that the desired generators generate all of the elements $z_1^{d_1} * z_1^{d_2} * z_1^{d_3}$ for $0\leq d_1, d_2, d_3\leq 2$ and the second part will be an induction argument to generalize this to $d_1, d_2, d_3\in\mathbb{Z}$. We will first prove the first part. Using Lemma \ref{lemma3:2}, we have 
    \begin{align*}
        1_1 * 1_1 * z_1 & = (z_1 + z_2 + z_3)1_1*1_1*1_1 - z_1 * 1_1*1_1-1_1*z_1*1_1,\\
        z_1 * 1_1 * z_1 & = (z_1 + z_2 + z_3)z_1*1_1*1_1 - z_1^2 * 1_1*1_1-z_1*z_1*1_1,\\
        1_1 * z_1 * z_1 & = (z_1^{-1} + z_2^{-1} + z_3^{-1})(z_1z_2z_3)1_1*1_1*1_1 - z_1 * 1_1*z_1-z_1*z_1*1_1,\\
        z_1^2 * 1_1 * z_1 & = (z_1^{-1} + z_2^{-1} + z_3^{-1})(z_1z_2z_3)z_1*1_1*1_1 - z_1^2 * z_1*1_1-(z_1z_2z_3)1_1*1_1*1_1.
    \end{align*}
    From here, we can see that $d_2$ and $d_3$ have the same restrictions, so we will omit equations that can be obtained from other equations by switching $d_2$ and $d_3$.
    \begin{align*}
        z_1^2 * z_1^2 * 1_1 & = (z_1^{-1} + z_2^{-1} + z_3^{-1})(z_1z_2z_3)z_1*z_1*1_1 - (z_1z_2z_3)z_1 * 1_1*1_1-(z_1z_2z_3)1_1*z_1*1_1,\\
        1_1 * z_1^2 * 1_1 & = (z_1 + z_2 + z_3)1_1*z_1*1_1 - z_1 * z_1*1_1-1_1*z_1*z_1,\\
        z_1 * z_1^2 * 1_1 & = (z_1 + z_2 + z_3)z_1*z_1*1_1 - z_1^2*z_1*1_1-(z_1z_2z_3)1_1*1_1*1_1,\\
        1_1 * z_1^2 * z_1 & = (z_1^{-1} + z_2^{-1} + z_3^{-1})(z_1z_2z_3)1_1*z_1*1_1 - z_1 * z_1^2*1_1-(z_1z_2z_3)1_1*1_1*1_1,\\
        1_1 * z_1^2 * z_1^2 & = (z_1^{-1} + z_2^{-1} + z_3^{-1})(z_1z_2z_3)1_1*z_1*z_1 - (z_1z_2z_3)1_1 * 1_1*z_1-(z_1z_2z_3)1_1*z_1*1_1.
    \end{align*}
    From here, all other elements of the desired form can be obtained by applying Lemma \ref{lemma3:2}\ref{lemma3:2:a} to existing elements, so we are done with the first part of the proof.
    
    For the second part of the proof, we will prove that we can generate all elements with $0\leq d_1, d_2, d_3\leq n$ for any $n$ using induction. Our base case is $n = 2$, which we proved in the first part of the proof. Now, we need to prove that we can generate all elements with $0\leq d_1, d_2, d_3\leq n + 1$. By symmetry, we only need to show that we can generate elements with $d_1 = n + 1$. Also note that if an element has $d_2, d_3>0$ then we can simply apply Lemma \ref{lemma3:2}\ref{lemma3:2:a} to generate it, so by symmetry we can also say that $d_3 = 0$. Now, if $0\leq d_2 < n$, then we have $$z_1^{n + 1} * z_1^{d_2} * 1_1 = (z_1 + z_2 + z_3)z_1^n * z_1^{d_2} * 1_1 - z_1^n * z_1^{d_2 + 1} * 1_1 - z_1^n * z_1^{d_2} * z_1.$$Similarly, if we have $1 < d_2\leq n + 1$ then we have $$z_1^{n + 1} * z_1^{d_2} * 1_1 = (z_1^{-1} + z_2^{-1} + z_3^{-1})(z_1z_2z_3)z_1^n * z_1^{d_2 - 1} * 1_1 - (z_1z_2z_3)z_1^n * z_1^{d_2 - 2} * 1_1 - (z_1z_2z_3)z_1^{n - 1} * z_1^{d_2 - 1} * 1_1.$$Since our base case is 2, we are only concerned with $n\geq 2$, so there is no $d_2$ with $n\leq d_2\leq1$. Therefore, we can generate all elements with $0\leq d_1, d_2, d_3\leq n + 1$, so we can generate all elements with $d_1, d_2, d_3$ nonnegative. Now, we can simply apply Lemma \ref{lemma3:2}\ref{lemma3:2:a} to generate all elements $z_1^{d_1}*z_1^{d_2}*z_1^{d_3}$ with $d_1, d_2, d_3\in\mathbb{Z}$.
\end{proof}

A question that naturally arises from this theorem is whether the generating sets shown in parts \ref{thm3:3:b} and \ref{thm3:3:c} are of minimal size. Although we cannot give a full answer to this question, the following remark shows a satisfactory answer in the context of this section.

\begin{remark}
    The number of generators of the form $z_1^{d_1}*\cdots*z_1^{d_k}$ needed to generate $A^\mathbf{R}_k$ using only the conditions in Proposition \ref{prop3:1} and Lemma \ref{lemma3:2} is bounded below by $k(k-1)$. To see why this is true, first assume that there is some set of generators $S$ that generates $A^\mathbf{R}_k$ such that $|S|<k(k-1)$. Let $S_i=\{z_1^{d_1}*\cdots*z_1^{d_k}\mid d_1+\cdots+d_k\equiv i\pmod k\}$. By the Pigeonhole Principle, $|S\cap S_i| < k-1$ for some $i$. We will show that $S$ cannot generate any elements of $S_i$ other than elements of the form $z_1^{d_1+n}*\cdots*z_1^{d_k+n}$ for some $z_1^{d_1}*\cdots*z_1^{d_k}\in S$ and $n\in\mathbb{Z}$. 
    
    First, let us consider Lemma \ref{lemma3:2}\ref{lemma3:2:a}. Note that the sum of the exponents of an element generated using this condition is $kn$ plus the sum of the exponents of the element used to generate it, for some integer $n$. Thus, as $k\mid kn$, both the generating and the generated element belong to the same $S_j$. If they are in $S_i$, then the generated element will simply be an element of the form $z_1^{d_1+n}*\cdots*z_1^{d_k+n}$ for some $z_1^{d_1}*\cdots*z_1^{d_k}\in S$. 
    
    Now we will consider Lemma \ref{lemma3:2}\ref{lemma3:2:b}. This condition requires the element $z_1^{d_1}*\cdots*z_1^{d_k}$ and $k-1$ of the elements $$z_1^{d_1+n}*z_1^{d_2}*\cdots*z_1^{d_k},z_1^{d_1}*z_1^{d_2+n}*z_1^{d_3}*\cdots*z_1^{d_k},\dots,z_1^{d_1}*\cdots*z_1^{d_k+n}$$to generate the last of those elements ($z_1^{d_1}*\cdots*z_1^{d_k}$ cannot be generated since we cannot invert $z_1^n + z_2^n + z_3^n + z_4^n$). In particular, the sums of the exponents of all of the $k$ aforementioned elements are the same, so to generate an element we need $k-1$ other elements with the same sum of exponents. Now, let us say we want to generate an element in $S_i$ whose sum of exponents is $s$. In order to generate the element we must have $k-1$ other elements whose sum of exponents is also $s$. However, for each $z_1^{d_1}*\cdots*z_1^{d_k}\in S\cap S_i$ there is a unique $n$ such that the sum of the exponents of $z_1^{d_1+n}*\cdots*z_1^{d_k+n}$ is $s$. Since $|S\cap S_i| < k-1$, this implies we have less than $k-1$ elements whose sum of exponents is $s$. Therefore, we cannot generate any elements of $S_i$ using Lemma \ref{lemma3:2}\ref{lemma3:2:b}.
\end{remark}

By the above remark, the generating sets found in Theorem \ref{thm3:3} are indeed the smallest possible sets that we can find using Proposition \ref{prop3:1} and Lemma \ref{lemma3:2}. The remark also prompts the question of whether the lower bounds can be achieved for $k\geq 4$. As the following remark shows, they are not achievable using only Proposition \ref{prop3:1} and Lemma \ref{lemma3:2}. Note that we do not prove that they are not achievable in general, as there may be additional relations between shuffle elements that allow the lower bounds to be achieved.

\begin{remark}
    Proposition \ref{prop3:1} and Lemma \ref{lemma3:2} do not necessarily imply that $A^{\mathbf{R}}_k$ is finitely generated as a module over $V_k$ for $k\geq 4$.
    To see this, note that by Proposition \ref{prop3:1}, any generators of $A^\mathbf{R}_k$ can be written as a linear combination of elements of the form $z_1^{d_1}*\cdots*z_1^{d_k}$, so if $A^\mathbf{R}_k$ is finitely generated then it must be finitely generated by generators of the form $z_1^{d_1}*\cdots*z_1^{d_k}$. Thus, we need only consider those generators.
    
    For simplicity of notation, we will only show the case $k = 4$. For $k > 4$, we can simply take the first four factors in the shuffle product and apply the proof.
    
    Let us first define the \textbf{range} of an element $z_1^{d_1} * z_1^{d_2} * z_1^{d_3} * z_1^{d_4}$ to be $$R(z_1^{d_1} * z_1^{d_2} * z_1^{d_3} * z_1^{d_4}) = \text{max}(d_{\sigma(1)} + d_{\sigma(2)} - d_{\sigma(3)} - d_{\sigma(4)})_{\sigma\in S_4}.$$Now, consider any finite set of generators and let $R_{max}$ be the maximum range of any generator. We will prove that Lemma \ref{lemma3:2} can only generate elements whose range is at most $R_{max}$. Let us say that we want to generate an element $z_1^{d_1} * z_1^{d_2} * z_1^{d_3} * z_1^{d_4}$ with $$R(z_1^{d_1} * z_1^{d_2} * z_1^{d_3} * z_1^{d_4}) > R_{max}.$$
    If we use Lemma \ref{lemma3:2}\ref{lemma3:2:a} we get $$z_1^{d_1} * z_1^{d_2} * z_1^{d_3} * z_1^{d_4} = (z_1z_2z_3z_4)^nz_1^{d_1 - n} * z_1^{d_2 - n} * z_1^{d_3 - n} * z_1^{d_4 - n},$$
    however we have 
    \begin{align*}
        R(z_1^{d_1 - n} * z_1^{d_2 - n} * z_1^{d_3 - n} * z_1^{d_4 - n}) & = \text{max}((d_{\sigma(1)}-n) + (d_{\sigma(2)}-n) - (d_{\sigma(3)}-n) - (d_{\sigma(4)}-n))_{\sigma\in S_4}\\
        & = \text{max}(d_{\sigma(1)} + d_{\sigma(2)} - d_{\sigma(3)} - d_{\sigma(4)})_{\sigma\in S_4} = R(z_1^{d_1} * z_1^{d_2} * z_1^{d_3} * z_1^{d_4})\\
        & > R_{max}.
    \end{align*}
    
    So Lemma \ref{lemma3:2}\ref{lemma3:2:a} cannot generate $z_1^{d_1} * z_1^{d_2} * z_1^{d_3} * z_1^{d_4}$ since $z_1^{d_1 - n} * z_1^{d_2 - n} * z_1^{d_3 - n} * z_1^{d_4 - n}$ cannot be a generator. If we use Lemma \ref{lemma3:2}\ref{lemma3:2:b} then since the choice of variable is not important, without loss of generality we can say that we have 
    \begin{multline*}
        z_1^{d_1} * z_1^{d_2} * z_1^{d_3} * z_1^{d_4}=(z_1^n + z_2^n + z_3^n + z_4^n)z_1^{d_1-n} * z_1^{d_2} * z_1^{d_3} * z_1^{d_4}\\
        - z_1^{d_1-n} * z_1^{d_2+n} * z_1^{d_3} * z_1^{d_4} - z_1^{d_1-n} * z_1^{d_2} * z_1^{d_3+n} * z_1^{d_4} - z_1^{d_1-n} * z_1^{d_2} * z_1^{d_3} * z_1^{d_4+n}.
    \end{multline*}
    
    Now, if the $d_1$ term is positive in the range of $z_1^{d_1} * z_1^{d_2} * z_1^{d_3} * z_1^{d_4}$, then if $n$ is nonpositive then $$R(z_1^{d_1-n} * z_1^{d_2} * z_1^{d_3} * z_1^{d_4}) = R(z_1^{d_1} * z_1^{d_2} * z_1^{d_3} * z_1^{d_4}) - n \geq R(z_1^{d_1} * z_1^{d_2} * z_1^{d_3} * z_1^{d_4}) > R_{max},$$so $z_1^{d_1-n} * z_1^{d_2} * z_1^{d_3} * z_1^{d_4}$ cannot be a generator. If $n$ is positive then without loss of generality assume that $d_2$ is also positive in the range of $z_1^{d_1} * z_1^{d_2} * z_1^{d_3} * z_1^{d_4}$. Then we have $$R(z_1^{d_1-n} * z_1^{d_2+n} * z_1^{d_3} * z_1^{d_4}) = R(z_1^{d_1-n} * z_1^{d_2} * z_1^{d_3} * z_1^{d_4}) + n \geq R(z_1^{d_1} * z_1^{d_2} * z_1^{d_3} * z_1^{d_4}) > R_{max},$$so $z_1^{d_1-n} * z_1^{d_2+n} * z_1^{d_3} * z_1^{d_4}$ cannot be a generator. A similar proof holds for the case where the $d_1$ term is negative. Therefore, no element with a range greater than $R_{max}$ can be generated, so no finite set of generators can generate all of $A^{\mathbf{R}}_4$ using only the conditions in Proposition \ref{prop3:1} and Lemma \ref{lemma3:2}.
\end{remark}

These two remarks show the limitations of Lemma \ref{lemma3:2}, however the approach of using the generators of the integral shuffle algebra to study its structure may still be useful. Lemma \ref{lemma3:2} does not cover all of the relations between the generators of $A^\mathbf{R}_k$ as a $V_k$-module. For example, since $A^\mathbf{R}_k\subset V_k$, the generators themselves are also scalars, leading to trivial relations such as $$(1_1*1_1)z_1*1_1=(z_1*1_1)1_1*1_1.$$
These relations do not follow from Lemma \ref{lemma3:2}. In general, the discovery and description of new relations may help to solve the questions of whether the generating sets in parts \ref{thm3:3:b} and \ref{thm3:3:c} of Theorem \ref{thm3:3} are of minimal size and whether $A^\mathbf{R}_k$ is finitely generated for $k\geq 4$. 

\section{Conditions for the general case}\label{sec4}
In this section we will discuss some necessary conditions for an element of $V$ to be in $A^{\mathbf{R}}$. We present these necessary conditions in the form of membership in an ideal of $V_k$. Recall the definition of the wheel conditions (\ref{prop2:5}):$$p(z_1, z_2, z_3, \dots) = 0\text{ whenever }\bigg\{\frac{z_1}{z_2}, \frac{z_2}{z_3}, \frac{z_3}{z_1}\bigg\} = \bigg\{q_1, q_2, \frac{1}{q}\bigg\},$$which Negut proved were necessary in \cite{negut2014shuffle}. Note that they can be rewritten in an ideal form:

\begin{proposition}[Ideal form of the wheel conditions]
    $A^{\mathbf{R}}_k$ is contained in the intersection of the ideals $$(q_1z_1-z_2,q_2z_2-z_3),\quad(q_2z_1-z_2,q_1z_2-z_3)$$ of $V_k$ for $k \geq 3$.
\end{proposition}
\begin{proof}
    Take the the elements of $A^\mathbf{R}_k$ modulo $(q_1z_1-z_2,q_2z_2-z_3)$. This is equivalent to setting $$q_1z_1-z_2=q_2z_2-z_3=0,$$
    or $$\frac{z_2}{z_1}=q_1,\quad \frac{z_3}{z_2}=q_2.$$
    Then we also have $$\frac{z_1}{z_3}=\frac{z_1}{z_2}\cdot\frac{z_2}{z_3}=\frac{1}{q_1}\cdot\frac{1}{q_2}=\frac{1}{q}.$$
    Therefore, by the wheel conditions, all elements of $A^\mathbf{R}_k$ are equivalent to 0 modulo $(q_1z_1-z_2,q_2z_2-z_3)$, so they are in the ideal. Similarly, taking the elements modulo $(q_2z_1-z_2,q_1z_2-z_3)$ is equivalent to setting $$\frac{z_2}{z_1}=q_2,\quad\frac{z_3}{z_2}=q_1,\quad\frac{z_1}{z_3}=\frac{1}{q}.$$
    Again, by the wheel conditions, all elements of $A^\mathbf{R}_k$ are equivalent to 0 modulo $(q_2z_1-z_2,q_1z_2-z_3)$, so they are in the ideal.
\end{proof}

To find our next necessary condition, we revisit $A^\mathbf{R}_2$, which we found in Theorem \ref{thm3:3}\ref{thm3:3:b} to be the $V_2$ module generated by $1_1*1_1$ and $z_1*1_1$, or equivalently the ideal $$(1_1*1_1,z_1*1_1)=$$$$(2qz_1^2 - (1 + q_1 + q_2 - 2q + q_1q + q_2q + q^2)z_1z_2 + 2qz_2^2, qz_1^3 + (-q_1 - q_2 + 2q - q^2)(z_1+z_2)z_1z_2 + qz_2^3)$$ of $V_2$. We may rewrite it in the simpler form $$\bigg(1_1*1_1,\frac{2(z_1*1_1)-(z_1+z_2)(1_1*1_1)}{z_1z_2}\bigg)=$$$$(2qz_1^2 - (1 + q_1 + q_2 - 2q + q_1q + q_2q + q^2)z_1z_2 + 2qz_2^2, (1 - q_1)(1 - q_2)(1 - q)(z_1 + z_2)),$$and we will extend this to general $A^\mathbf{R}_k$ in the next theorem.

\begin{theorem}\label{thm4:2}
    $A^{\mathbf{R}}_k$ is contained in the ideal $$(2qz_1^2 - (1 + q_1 + q_2 - 2q + q_1q + q_2q + q^2)z_1z_2 + 2qz_2^2, (1 - q_1)(1 - q_2)(1 - q)(z_1 + z_2))$$of $V_k$ for $k \geq 2$.
\end{theorem}
\begin{proof}
    We will show that the shuffle product of 2 elements in the ideal (or elements in $A^\mathbf{R}_1$) is also in the ideal, which will imply the desired result. Consider the shuffle product of any two shuffle elements:$$P(z_1, \dots, z_k) * Q(z_1, \dots, z_l) = \frac{1}{k!l!}\text{Sym} \bigg[P(z_1, \dots, z_k)Q(z_{k + 1}, \dots, z_{k + l})\prod_{1 \leq i \leq k < j \leq k + l}\omega(z_i,z_j)\bigg].$$We will prove that every summand of the Sym is contained in the desired ideal. For any given summand, let the symmetric operator take $z_1$ and $z_2$ to $z_m$ and $z_n$, respectively. If $m$ and $n$ are either both in $\{1, \dots, k\}$ or both in $\{k + 1, \dots, k + l\}$ then the summand will be in the ideal since the corresponding factor will have at least 2 variables and will therefore be in the ideal. If $m$ is in $\{1, \dots, k\}$ and $n$ is in $\{k + 1, \dots, k + l\}$, then the product $$\prod_{1 \leq i \leq k < j \leq k + l}\frac{(z_i - qz_j)(z_j - q_1z_i)(z_j - q_2z_i)}{z_i - z_j}$$contains the term $$\frac{(z_m - qz_n)(z_n - q_1z_m)(z_n - q_2z_m)}{z_m - z_n}$$$$= \frac{1}{2}(2qz_m^2 - (1 + q_1 + q_2 - 2q + q_1q + q_2q + q^2)z_mz_n + 2qz_n^2) + \frac{z_mz_n}{2(z_m - z_n)}(1 - q_1)(1 - q_2)(1 - q)(z_m + z_n),$$so the summand is in the ideal. Similarly, if $n$ is in $\{1, \dots, k\}$ and $m$ is in $\{k + 1, \dots, k + l\}$, then the summand contains the term $$\frac{(z_n - qz_m)(z_m - q_1z_n)(z_m - q_2z_n)}{z_n - z_m}$$$$= \frac{1}{2}(2qz_m^2 - (1 + q_1 + q_2 - 2q + q_1q + q_2q + q^2)z_mz_n + 2qz_n^2) + \frac{z_mz_n}{2(z_n - z_m)}(1 - q_1)(1 - q_2)(1 - q)(z_m + z_n),$$so the summand is in the ideal.\end{proof}
    
    Two open problems naturally follow from the above results. The first is whether a similar result to Theorem \ref{thm4:2} can be derived from the six generators of $A^\mathbf{R}_3$ described in Theorem \ref{thm3:3}\ref{thm3:3:c}. The second is whether the intersection of the ideal form of the wheel conditions and the ideal in Theorem \ref{thm4:2} is equal to the ideal generated by the six generators of $A^\mathbf{R}_3$ when considered over $V_k$ for $k\geq 3$. If this is true, then the wheel conditions and Theorem \ref{thm4:2} would form necessary and sufficient conditions for $A^\mathbf{R}_3$, which would be a remarkable result.
    
    We conclude this section with the following corollary, which presents an observation that may be useful in future considerations of the integral shuffle algebra.
    
\begin{corollary}
    Every element $P(z_1, z_2, \dots, z_k)\in A^{\mathbf{R}}$ must satisfy$$P(z_1, -z_1,\dots, z_k) = c(1 + q_1)(1 + q_2)(1 + q)$$for some $c\in V_k$.
\end{corollary}
\begin{proof}
    By Theorem \ref{thm4:2}, the integral shuffle algebra is contained in the ideal $$(2qz_1^2 - (1 + q_1 + q_2 - 2q + q_1q + q_2q + q^2)z_1z_2 + 2qz_2^2, (1 - q_1)(1 - q_2)(1 - q)(z_1 + z_2)).$$Therefore, we only need to prove that the two generators of the ideal satisfy the above condition. Indeed, the generators become $$2qz_1^2 - (1 + q_1 + q_2 - 2q + q_1q + q_2q + q^2)z_1(-z_1) + 2q(-z_1)^2 = z_1^2(1 + q_1)(1 + q_2)(1 + q)$$and$$(1 - q_1)(1 - q_2)(1 - q)(z_1 - z_1) = 0$$when we substitute $z_2 = -z_1$, which suffices for the proof.\end{proof}
    
\section{Acknowledgements}
I would like to thank my mentor, Yu Zhao, for introducing me to this project and providing helpful guidance throughout the research process. I would also like to thank Yongyi Chen and Dr.~Tanya Khovanova for providing suggestions and proofreading this paper. I also thank the MIT Math Department and the PRIMES-USA program for providing me with this wonderful research opportunity.\newpage
\printbibliography
\end{document}